\documentclass[english]{amsart}
\usepackage[utf8]{inputenc}
\usepackage{babel}
\usepackage{amstext}
\usepackage{amsthm}
\usepackage{amssymb}
\usepackage{xargs}[2008/03/08]
\usepackage[all]{xy}
\usepackage[unicode=true,
 bookmarks=false,
 breaklinks=false,pdfborder={0 0 1},backref=section,colorlinks=false]
 {hyperref}

\makeatletter
\numberwithin{equation}{section}
\numberwithin{figure}{section}
\theoremstyle{plain}
\newtheorem{thm}{\protect\theoremname}
\theoremstyle{plain}
\newtheorem{lem}[thm]{\protect\lemmaname}
\theoremstyle{remark}
\newtheorem{rem}[thm]{\protect\remarkname}
\theoremstyle{plain}
\newtheorem{cor}[thm]{\protect\corollaryname}
\theoremstyle{plain}
\newtheorem{prop}[thm]{\protect\propositionname}

\@ifundefined{date}{}{\date{}}
\usepackage{enumerate}\usepackage{xcolor}\usepackage{amscd}
\usepackage{xargs}[2008/03/08]

\usepackage[all]{xy}
\usepackage{xcolor}

\usepackage{tikz}

\usepackage{pgfplots}

\usetikzlibrary{arrows}
\usetikzlibrary{3d}

\usetikzlibrary{automata}

\usetikzlibrary{chains}

\usetikzlibrary{positioning}

\usetikzlibrary{backgrounds}

\usetikzlibrary{calc,decorations.pathreplacing}

\pgfplotsset{compat=1.9}

\usepackage{enumitem}
\setenumerate{ label= (\alph*)}
\setenumerate[2]{ label= (\alph{enumi}\arabic*)} 

\newcommand{\pullbackcorner}[1][dr]{\save*!/#1+1.2pc/#1:(1,-1)@^{*}\restore}
\newcommand{\pushoutcorner}[1][dr]{\save*!/#1-1.2pc/#1:(-1,1)@^{*}\restore}

\def\mod {\mathop{\rm mod}\nolimits}

\def\Ker {\mathop{\rm Ker}\nolimits}

\def\B{\mathcal B}

\def\t{\mathbf{t}}

\def\Coker{\mathrm{Coker}}

\def\A{\mathcal {A}}

\def\C{\mathcal C}

\def\colim{\mathsf{colim}_\Sigma}

\global\long\def\Fun{\operatorname{Fun}}
\global\long\def\A{\mathcal{A}}
\global\long\def\B{\mathcal{B}}
\global\long\def\col{\mathsf{colim}}
\global\long\def\limite{\mathsf{lim}}

\@namedef{subjclassname@2020}{%
 \textup{2020} Mathematics Subject Classification}

\makeatother

\providecommand{\corollaryname}{Corollary}
\providecommand{\lemmaname}{Lemma}
\providecommand{\propositionname}{Proposition}
\providecommand{\remarkname}{Remark}
\providecommand{\theoremname}{Theorem}

\begin{document}
\title{Exactness of limits and colimits in abelian categories revisited}
\author{Alejandro Argudín-Monroy}
\address{Instituto de Matemáticas, Universidad Nacional Autónoma de México,
Circuito exterior, Ciudad Universitaria, CDMX 04510, México}
\email{argudin@ciencias.unam.mx}
\thanks{The first named author was supported by a postdoctoral fellowship
from Programa de Desarrollo de las Ciencias Básicas, Ministerio de
educación y cultura, Universidad de la República, Uruguay. He is currently
supported with a by a postdoctoral fellowship EPM(1) 2024 from SECIHTI}
\author{Carlos E. Parra}
\address{Instituto de Ciencias F{\'i}sicas y Matemáticas, Edificio Emilio
Pugin, Campus Isla Teja, Universidad Austral de Chile, 5090000 Valdivia,
CHILE}
\email{carlos.parra@uach.cl}
\thanks{The second named author was supported by CONICYT/FONDECYT/REGULAR/1200090.
He is currently supported by ANID+FONDECYT/REGULAR+1240253}
\begin{abstract}
Let $\Sigma$ be a small category and $\mathcal{A}$ be a $\Sigma$-co-complete
(resp. $\Sigma$-complete) abelian category. It is a well-known fact
that the category $\operatorname{Fun}(\Sigma,\mathcal{A})$ of functors
of $\Sigma$ in $\mathcal{A}$ is an abelian category, and that the
functor $\mathsf{colim}_{\Sigma}(-):\operatorname{Fun}(\Sigma,\mathcal{A})\rightarrow\mathcal{A}$
(resp. $\mathsf{lim}_{\Sigma}(-):\operatorname{Fun}(\Sigma,\mathcal{A})\rightarrow\mathcal{A}$)
is left (resp. right) adjoint to $\kappa^{\Sigma}:\mathcal{A}\rightarrow\operatorname{Fun}(\Sigma,\mathcal{A})$,
where $\kappa^{\Sigma}$ is the associated constant diagram functor.
In this paper we will show that the functor $\mathsf{colim}_{\Sigma}(-)$
(resp. $\mathsf{lim}_{\Sigma}(-)$) is exact if and only if the pair
of functors $\left(\mathsf{colim}_{\Sigma}(-),\kappa^{\Sigma}\right)$
(resp. $\left(\kappa^{\Sigma},\mathsf{lim}_{\Sigma}(-)\right)$) is
Ext-adjoint. As an application of our findings, we will give new proofs
of known results on the exactness of limits and colimits in abelian
categories. In particular, we will prove that every complete category
with enough projective effacements has exact products; and that a
necessary condition for the heart of the t-structure associated to
a torsion pair $(\mathcal{T},\mathcal{F})$ to have exact direct limits
is that $\mathcal{F}$ is closed under direct limits, whenever the
ambient abelian category has exact direct limits. 
\end{abstract}

\maketitle
\setcounter{tocdepth}{1} \tableofcontents{}

\section{Introduction}

\newcommandx\suc[5][usedefault, addprefix=\global, 1=N, 2=M, 3=K, 4=, 5=]{#1\overset{#4}{\hookrightarrow}#2\overset{#5}{\twoheadrightarrow}#3}%

\global\long\def\Ext{\mathrm{Ext}}%

\global\long\def\Hom{\mathrm{Hom}}%

\global\long\def\Mod{\mathrm{Mod}}%

\global\long\def\mod{\mathrm{mod}}%

\global\long\def\Coker{\mathrm{Coker}}%

\global\long\def\Ker{\mathrm{Ker}}%

\global\long\def\im{\mathrm{Im}}%

\global\long\def\t{\mathfrak{t}}%

\global\long\def\s{[0]}%

\global\long\def\ss{[1]}%

Let $\mathcal{A}$ be an abelian category. Recall that $\mathcal{A}$
is Ab3 if arbitrary coproducts exist; it is Ab4 if it is Ab3 and {every}
coproduct of a set of monomorphisms is a monomorphism; and it is Ab5
if it is Ab3 and {every} colimit of a direct system of monomorphisms
is a monomorphism. The corresponding dual notions are known as Ab3{*},
Ab4{*}, and Ab5{*}. This classification of abelian categories was
introduced by Grothendieck in \cite{Ab} in order to axiomatise the
constructions of the homological algebra coming from categories of
modules.

Given $X,Y\in\mathcal{A}$, consider the `big' group of extensions
$\Ext_{\A}^{1}(X,Y)$ (see \cite[Chap. VII]{mitchell}). To get a
glimpse of the essence of these groups, recall that $\Ext_{\A}^{1}(X,Y)$
is made up of the equivalence classes of the short exact sequences
of the form $Y\hookrightarrow E\twoheadrightarrow X$. So that, for
every exact sequence $\epsilon:Y\hookrightarrow E\twoheadrightarrow X$
in $\mathcal{A}$, we will denote as $\overline{\epsilon}$ the extension
in $\Ext_{\A}^{1}(X,Y)$ representing $\epsilon$. Now, recall that
morphisms in $\mathcal{A}$ induce morphisms between extension groups
via pullback and pushout. Namely, every $f\in\Hom_{\A}(X',X)$ induces
the additive maps $\Ext_{\A}^{1}(f,Y):\Ext_{\A}^{1}(X,Y)\rightarrow\Ext_{\A}^{1}(X',Y)$,
$\overline{\eta}\mapsto\overline{\eta}\cdot f$, and $\Ext_{\A}^{1}(Y,f):\Ext_{\A}^{1}(Y,X')\rightarrow\Ext_{\A}^{1}(Y,X)$,
$\overline{\eta}\mapsto f\cdot\overline{\eta}$, for every $Y\in\mathcal{A}$.

Now, for every $A\in\mathcal{A}$ and for every coproduct $\coprod_{i\in X}B_{i}$
(resp. product $\prod_{i\in X}B_{i}$) in $\mathcal{A}$, consider
the map $\Psi:\Ext_{\A}^{1}(\coprod_{i\in X}B_{i},A)\rightarrow\prod_{i\in X}\Ext_{\A}^{1}(B_{i},A)$,
$\overline{\eta}\mapsto\left(\overline{\eta}\cdot\mu_{i}\right)_{i\in X}$
(resp. $\Phi:\Ext_{\A}^{1}(A,\prod_{i\in X}B_{i})\rightarrow\prod_{i\in X}\Ext_{\A}^{1}(A,B_{i})$,
$\overline{\eta}\mapsto\left(\pi_{i}\cdot\overline{\eta}\right)_{i\in X}$),
where $\mu_{i}:B_{i}\rightarrow\coprod_{i\in X}B_{i}$ is the $i$-th
canonical inclusion (resp. $\pi_{i}:\prod_{i\in X}B_{i}\rightarrow B_{i}$
the $i$-th projection). In \cite{argudin2021yoneda}, it was shown
that the Ab4 (resp. Ab4{*}) categories are exactly the Ab3 (resp.
Ab3{*}) categories where the maps $\Psi$ (resp. $\Phi$) are bijective.
A key point of this work is that, for every set of exact sequences
$H=\left\{ \eta_{i}:\:A\stackrel{f_{i}}{\hookrightarrow}E_{i}\twoheadrightarrow B_{i}\right\} _{i\in X}$
(resp. $H=\left\{ \eta_{i}:\:B_{i}\hookrightarrow E_{i}\stackrel{f_{i}}{\twoheadrightarrow}A\right\} _{i\in X}$),
it can be built an exact sequence 
\[
\Theta(H):\:A\rightarrow E\twoheadrightarrow\coprod_{i\in X}A_{i}\qquad(\mbox{resp. }\Theta(H):\:\prod_{i\in X}B_{i}\hookrightarrow E\rightarrow A)\mbox{,}
\]
where $E$ is the colimit (resp. limit) of a diagram defined through
the family of morphisms $\{f_{i}\}_{i\in X}$ (see \cite[Sec. 4]{argudin2021yoneda}).
It turns out that, under the right conditions, the inverse map of
$\Psi$ (resp. $\Phi$) can be exhibited using $\Theta$.

Let $\Sigma$ be a small category and $\operatorname{Fun}(\Sigma,\mathcal{A})$
be the category of functors $\Sigma\rightarrow\mathcal{A}$. Recall
that $\mathcal{A}$ is $\Sigma$-co-complete (resp. $\Sigma$-complete)
if every $F\in\operatorname{Fun}(\Sigma,\mathcal{A})$ admits a colimit
(resp. limit) in $\mathcal{A}$. In this context it is immediate to
see that one has an adjoint pair $\left(\colim(-),\kappa^{\Sigma}\right)$
(resp. $\left(\kappa^{\Sigma},\limite_{\Sigma}(-)\right)$), where
$\colim(-):\operatorname{Fun}(\Sigma,\mathcal{A})\rightarrow\mathcal{A}$
(resp. $\limite_{\Sigma}(-):\operatorname{Fun}(\Sigma,\mathcal{A})\rightarrow\mathcal{A}$)
is the colimit (resp. limit) functor, and where $\kappa^{\Sigma}:\mathcal{A}\rightarrow\operatorname{Fun}(\Sigma,\mathcal{A})$
is the functor that assigns to each $A\in\mathcal{A}$ a constant
functor in $\operatorname{Fun}(\Sigma,\mathcal{A})$.

In this article, motivated mainly by {the sequence }$\Theta$, we
will revisit the techniques used in \cite{argudin2021yoneda} to study
the exactness of {$\Sigma$-}colimits (resp. {$\Sigma$-}limits)
in {$\Sigma$-}co-complete (resp. {$\Sigma$-}complete) abelian
categories, for all small category $\Sigma$. Namely, we will consider
the canonical map 
\begin{gather*}
\Psi_{F,A}^{\Sigma}:\operatorname{Ext}_{\mathcal{A}}^{1}(\colim(F),A)\rightarrow\Ext_{\Fun(\Sigma,\A)}^{1}(F,\kappa^{\Sigma}(A))\\
\left(\mbox{resp. }\Phi_{F,A}^{\Sigma}:\operatorname{Ext}_{\mathcal{A}}^{1}(A,\limite_{\Sigma}(F))\rightarrow\Ext_{\Fun(\Sigma,\A)}^{1}(\kappa^{\Sigma}(A),F)\right)
\end{gather*}
for every $A\in\mathcal{A}$ and $F\in\operatorname{Fun}(\Sigma,\mathcal{A})$
(see Remark \ref{rem:defPsi} and Corollary \ref{cor:lim-exac}).
Looking for sufficient conditions for $\Psi_{F,A}^{\Sigma}$ (resp.
$\Phi_{F,A}^{\Sigma}$) to be bijective, we will see that every exact
sequence $\eta:\:\kappa^{\Sigma}(A)\hookrightarrow E\twoheadrightarrow F$
(resp. $\eta:\:F\hookrightarrow E\twoheadrightarrow\kappa^{\Sigma}(A)$)
in $\operatorname{Fun}(\Sigma,\mathcal{A})$ induces an exact sequence
$\kappa^{\Sigma_{*}}(A)\overset{f_{\eta}}{\rightarrow}E_{\eta}\twoheadrightarrow F'$
(resp. $F'\hookrightarrow E_{{\eta}}\overset{f_{\eta}}{\rightarrow}{\kappa^{\Sigma_{*}}(A)}$)
in $\operatorname{Fun}(\Sigma_{*},\mathcal{A})$, where $\Sigma_{*}$
is the one-point extension of $\Sigma$ that adds a source (resp.
sink) point $*$ (see Section 3), and $F'$ is a functor such that
$F'\restriction_{\Sigma}=F$ and $F'(*)=0$.

Our main result will be that, in a {$\Sigma$-}co-complete (resp.
{$\Sigma$-}complete) category $\mathcal{A}$, $\colim(-)$ (resp.
$\limite_{\Sigma}(-)$) is exact if and only if $f_{\eta}$ is a monomorphism
(resp. epimorphism) $\forall\overline{\eta}\in\Ext_{\Fun(\Sigma,\A)}^{1}(\kappa^{\Sigma}(A),F)$
(resp. $\forall\overline{\eta}\in\Ext_{\Fun(\Sigma,\A)}^{1}(F,\kappa^{\Sigma}(A))$),
which in turn will be equivalent to $\Psi_{F,A}^{\Sigma}$ (resp.
$\Phi_{F,A}^{\Sigma}$) being bijective $\forall A\in{\mathcal{A}}$
and $\forall F\in\operatorname{Fun}(\Sigma,\mathcal{A})$. Moreover,
in such case, it can be seen that $\Psi_{-,?}^{\Sigma}$ (resp. $\Phi_{-,?}^{\Sigma}$)
is a natural isomorphism {between the respective functors}. This
phenomenon is what we call an Ext-adjoint pair. An application of
our results consists in reproducing with new proofs the following
known results in abelian categories. 
\begin{enumerate}
\item \cite[Cor. 3.2.9]{Popescu} Let $\A$ be an Ab3 (resp. Ab3{*}) abelian
category. Then, $\A$ is Ab4 (resp. Ab4{*}) whenever $\A$ has enough
injectives {(resp. projectives)}. 
\item \cite[Thms 4.9 and 4.11]{argudin2021yoneda} Let $\A$ be an Ab3 (resp.
Ab3{*}) abelian category. Then, $\A$ is Ab4 (resp. Ab4{*}) if, and
only if, the canonical map 
\begin{gather*}
\text{Ext}_{\A}^{1}(\coprod_{i\in I}A_{i},A)\to\prod_{i\in I}\text{Ext}_{\A}^{1}(A_{i},A)\\
\left(\mbox{resp. }\text{Ext}_{\B}^{1}(B,\prod_{i\in I}B_{i})\to\prod_{i\in I}\text{Ext}_{\B}^{1}(B,B_{i})\right)
\end{gather*}
is always an isomorphism {of `big' abelian groups}. 
\item \cite[Cor. 1.4]{R} Let $\mathcal{A}$ be an Ab3{*} abelian category.
If $\mathcal{A}$ has enough projective effacements, then $\mathcal{A}$
is Ab4{*}. 
\item \cite[Thm. 4.8]{PS1} Let $\mathcal{A}$ be an Ab3{*} and Ab5 abelian
category. For a torsion pair $\t=(\mathcal{T},\mathcal{F})$ in $\mathcal{A}$,
consider the abelian category $\mathcal{H}_{\t}$ defined as the heart
of the Happel-Reiten-Smal{\o} t-structure associated to $\t$. If
$\mathcal{H}_{\t}$ is Ab5, then $\mathcal{F}$ is closed under direct
limits. 
\end{enumerate}
It is worth to point out that, we only know of two other proofs of
the statement in item (c) in the context of Grothendieck categories
(see \cite[Thm. B]{AP}, \cite[Thm. A]{B} or \cite[Cor. 1.4]{R});
and that this paper presents the third proof that can be found in
the literature of item (d) (see \cite[Thm. 4.8]{PS1} and \cite[Thm. 4.18]{PS5}).

The article is organized as follows. In Section 2 we recall the preliminaries
on abelian categories and (co)limits in abelian categories that we
will need throughout the article. In Section 3 we show how the one-point
extension of a small category $\Sigma$ is related to the behavior
of the colimits on the category $\Sigma$. Finally, Section 4 contains
our main results.

\bigskip{}

\section{Preliminaries}

\noindent In this section we recall some definitions and results,
and we fix some notation that we will use throughout the paper. For
more precise details, the reader is referred to \cite{Popescu,ringsofQuotients,maclane}.

\subsection{Limits and colimits}

Recall that a category is \textbf{small} when the isomorphism classes
of its objects form a set. In this sense, every set $\Sigma$ can
be viewed as a small category whose objects are the elements of $\Sigma$
and the only morphisms are the identity morphisms. Also, if $\Sigma$
is a directed (resp. codirected) set, then $\Sigma$ can be viewed
as a small category whose objects are the elements of $\Sigma$ and
there is a unique morphism $\alpha\to\beta$ exactly when $\alpha\leq\beta$
(resp. $\beta\leq\alpha$).

In what follows $\Sigma$ will denote a small category. If $\C$ is
a category, then a functor $\Sigma\to\C$ will be called a \textbf{$\Sigma$-diagram}
on $\C$. In this setting, the $\Sigma$-diagrams on $\C$ together
with the respective natural transformations form a category, which
will be denoted by $\Fun(\Sigma,\C).$ In particular, if $\mathcal{C}$
is an abelian category, then $\Fun(\Sigma,\mathcal{C})$ is an abelian
category in the obvious way (see \cite[Chapter IV, Section 7]{ringsofQuotients}).

Observe that the assignment $C\mapsto\kappa^{\Sigma}(C)$ gives a
functor $\kappa^{\Sigma}:\C\to\Fun(\Sigma,\C)$, where $\kappa^{\Sigma}(C)$
is the $\Sigma$-diagram on $\C$ such that $i\mapsto C$ and $\lambda\mapsto1_{C}$,
for each object $i$ and morphism $\lambda$ of $\Sigma$; and, for
every morphism $f:A\rightarrow B$ in $\mathcal{C}$, $\kappa^{\Sigma}(f):\kappa^{\Sigma}(A)\rightarrow\kappa^{\Sigma}(B)$
is the natural transformation $\kappa_{f}^{\Sigma}$ given by the
family of morphisms $(\kappa_{f,s}^{\Sigma})_{s\in\Sigma}$ with $\kappa_{f,s}^{\Sigma}=f$
for all $s\in\Sigma$. Such functor is called the \textbf{constant
diagram functor}.

Let $F\in\Fun(\Sigma,\mathcal{C})$, $C,L\in\mathcal{C}$, and $\rho:F\to\kappa^{\Sigma}(C)$,
$\varrho:\kappa^{\Sigma}(L)\to F$ be natural transformations. Then,
we will say that the pair $(C,\rho)$ (resp. $(L,\varrho)$) is a
\textbf{$\Sigma$-colimit} (resp. \textbf{$\Sigma$-limit}) of $F$
when the following condition holds: for each natural transformation
$\tau:{F}\to\kappa^{\Sigma}(D)$ (resp. $\tau:\kappa^{\Sigma}(D)\to{F}$),
where $D$ is an object in $\C$, there is a unique morphism $f:C\to D$
(resp. $f:D\to L$) such that the following diagram commutes.

\[
\xymatrix{ & \kappa^{\Sigma}(C)\ar[dd]^{\kappa_{f}^{\Sigma}} &  &  & \kappa^{\Sigma}(L)\ar[dl]_{\varrho}\\
F\ar[dr]^{\tau}\ar[ur]^{\rho} &  & (resp.) & F\\
 & \kappa^{\Sigma}(D) &  &  & \kappa^{\Sigma}(D)\ar[ul]^{\tau}\ar[uu]_{\kappa_{f}^{\Sigma}}
}
\]

If $(C,\rho)$ (resp. $(L,\varrho)$) is a $\Sigma$-colimit (resp.
$\Sigma$-limit) of $F$, we use the notation $C=\col_{\Sigma}(F)$
(resp. $L=\limite_{\Sigma}(F)$). Now, if each $\Sigma$-diagram on
$\C$ has a $\Sigma$-colimit (resp. $\Sigma$-limit), we say that
$\C$ \textbf{admits $\Sigma$-colimits} (resp. \textbf{$\Sigma$-limits})
{or that it is \textbf{$\Sigma$-co-complete} (resp. \textbf{$\Sigma$-complete})}.
In this case, the assignment $F\mapsto\col_{\Sigma}(F)$ (resp. $F\mapsto\limite_{\Sigma}(F)$)
gives rise to a functor $\col_{\Sigma}:\Fun(\Sigma,\C)\to\C$ (resp.
$\limite_{\Sigma}:\Fun(\Sigma,\C)\to\C$) which is the left (resp.
right) adjoint to the {associated} constant diagram functor. The
category $\C$ is called \textbf{co-complete} (resp. \textbf{complete})
when it admits $\Sigma$-colimits (resp. $\Sigma$-limits), for every
small category $\Sigma$. In this context, for every functor $F:\Sigma\to\mathcal{C}$,
we denote by $C_{F}:=\col_{\Sigma}(F)$ (resp. $L_{F}:=\lim_{\Sigma}(F)$)
and by $\rho^{F}:F\to\kappa^{\Sigma}(C_{F})$ (resp. $\varrho^{F}:\kappa^{\Sigma}(L_{F})\to F$)
the associated natural transformation.

Lastly, if $\mathcal{C}$ admits $\Sigma$-colimits (resp. $\Sigma$-limits),
observe that for each $A\in\A$ there is a unique morphism $\nabla^{A}:C_{\kappa^{\Sigma}(A)}\to A$
(resp. $\Delta^{A}:A\to L_{\kappa^{\Sigma}(A)}$) such that $\kappa_{\nabla^{A}}^{\Sigma}\circ\rho^{\kappa^{\Sigma}(A)}=1_{\kappa^{\Sigma}(A)}$
(resp. $\varrho^{\kappa^{\Sigma}(A)}\circ\kappa_{\Delta^{A}}^{\Sigma}=1_{\kappa^{\Sigma}(A)}$).
We refer to such morphism as the \textbf{co-diagonal morphism} (resp.
\textbf{diagonal morphism}).

\subsection{Abelian categories}

In what follows $\A$ will denote an abelian category. Recall the
following hierarchy among abelian categories (introduced by Grothendieck
in \cite{Ab}): we say that $\A$ is 
\begin{itemize}
\item \textbf{Ab3} if all set-indexed coproducts exist in $\A$ (equivalently,
if it is co-complete); 
\item \textbf{Ab4 }if it is Ab3 and the functors $\col_{\Sigma}$ are exact,
for each set $\Sigma$ viewed as a small category; 
\item \textbf{Ab5 }if it is Ab3 and the functors $\col_{\Sigma}$ are exact,
for each directed set $\Sigma$ viewed as a small category. 
\end{itemize}
We will denote by \textbf{Abn{*}} to the dual definition of \textbf{Abn}
for each $n$ in $\{3,4,5\}$.



\section{A one-point extension of a small category}

{In this section, $\Sigma$ will denote a small category and $\A$
will denote a $\Sigma$-co-complete abelian category. Now, f}rom
$\Sigma$ we can construct a new category $\Sigma_{*}$ which essentially
consists of adding a source point to $\Sigma$. That is, $\Sigma_{*}$
is the category consisting of the objects of $\Sigma$ plus a new
object $*$, and the morphisms of $\Sigma_{*}$ are the morphisms
of $\Sigma$ plus a single morphism from $*$ to any other point of
$\Sigma_{*}$. More precisely, we define the category $\Sigma_{*}$
as follows: 
\begin{itemize}
\item $Ob(\Sigma_{*}):=Ob(\Sigma)\cup\{*\}$; 
\item Let $(\alpha_{i})_{i\in Ob(\Sigma)}$ be an $Ob(\Sigma)$-indexed
family of symbols, where $\alpha_{i}=\alpha_{j}$ with $i,j\in Ob(\Sigma)$,
when $i$ is isomorphic to $j$ in $\Sigma$. In this setting, we
put $\Sigma_{*}(i,*):=\emptyset$, $\Sigma_{*}(*,i):=\{\alpha_{i}\}$,
$\Sigma_{*}(*,*):=\{1_{*}\}$ and $\Sigma_{*}(i,j)=\Sigma(i,j)$ for
every $i$ and $j$ objects in $\Sigma$. 
\item The composition is given by the extension of the composition in the
category $\Sigma$ following the rule: $1_{*}\circ1_{*}:=1_{*}$,
$\alpha_{i}\circ1_{*}:=\alpha_{i}$ and $\lambda\circ\alpha_{i}:=\alpha_{j}$,
for all $i,j$ and $\lambda:i\to j$, objects and morphism in $\Sigma$,
respectively. 
\end{itemize}

\subsection{A one-point extension of a $\Sigma$-diagram}

Let $\eta:\;\kappa^{\Sigma}(A)\stackrel{\phi}{\hookrightarrow}F\stackrel{\psi}{\twoheadrightarrow}G$
be an exact sequence in $\Fun(\Sigma,\A)$. Define the functor $F_{\eta}:\Sigma_{*}\to\A$
as follows: 
\begin{itemize}
\item For each $i\in Ob(\Sigma)$ we assign the object $F(i)$, and to the
object $*$ we assign the object $A$. 
\item For each morphism $\lambda$ in $\Sigma$ we assign the morphism $F(\lambda)$
and, for each $\alpha_{i}$ we assign the morphism $\phi_{i}$. Lastly,
for the morphism $1_{*}$ we assign the morphism $1_{A}$. 
\end{itemize}
It can be shown that there is an exact sequence $\eta':\;\kappa^{\Sigma_{*}}(A)\stackrel{\phi'}{\hookrightarrow}F_{\eta}\stackrel{\psi'}{\twoheadrightarrow}G'$
{in $\Fun(\Sigma_{*},\A)$}, where $G'$ is a functor that extends
$G$ such that $G'(*)=0$. In order to find the colimit of such exact
sequence, we proceed as follows.

By the right exactness of the functor $\col_{\Sigma}:\Fun(\Sigma,\A)\to\A$
(see \cite[Chapter V, Section 1]{ringsofQuotients}), there is an
exact sequence in $\A$: 
\[
\xymatrix{C_{\kappa^{\Sigma}(A)}\ar[rr]^{\col_{\Sigma}(\phi)} &  & C_{F}\ar[rr]^{\col_{\Sigma}(\psi)} &  & C_{G}\ar[r] & 0.}
\]
Consider the following commutative diagram in $\A$ with exact rows,
where the left square is the pushout of $\col_{\Sigma}(\phi)$ and
the co-diagonal morphism. 
\[
\xymatrix{C_{\kappa^{\Sigma}(A)}\ar[rr]^{\col_{\Sigma}(\phi)}\ar[d]_{\nabla^{A}} &  & C_{F}\ar[rr]^{\col_{\Sigma}(\psi)}\ar[d]_{\mu_{\eta}} &  & C_{G}\ar[r]\ar@{=}[d] & 0 & (1.1)\\
A\ar[rr]^{f_{\eta}} &  & Z_{\eta}\pullbackcorner\ar[rr]^{g_{\eta}} &  & C_{G}\ar[r] & 0
}
\]

Lastly, consider the morphism $\xi^{\eta}:F_{\eta}\to\kappa^{\Sigma_{*}}(Z_{\eta})$
in $\Fun(\Sigma_{*},\A)$ given by $\xi_{i}^{\eta}:=\mu_{\eta}\circ\rho_{i}^{F}$
for each $i\in Ob(\Sigma)$ and $\xi_{*}^{\eta}:=f_{\eta}$. Notice
that $\xi^{\eta}$ is well-defined since $\mu_{\eta}\circ\rho_{j}^{F}\circ F_{\eta}(\lambda)=\mu_{\eta}\circ\rho_{j}^{F}\circ F(\lambda)=\mu_{\eta}\circ\rho_{i}^{F}$
for each $\lambda:i\to j$ morphism in $\Sigma$, and the following
diagram in $\A$ is commutative for all $i\in Ob(\Sigma)$: 
\[
\xymatrix{F_{\eta}(*)=A\ar[rr]^{\xi_{*}^{\eta}=f_{\eta}}\ar[d]_{F_{\eta}(\alpha_{i})=\phi_{i}} &  & (\kappa^{\Sigma}(Z_{\eta}))(*)=Z_{\eta}\ar[d]^{(\kappa^{\Sigma}(Z_{\eta}))(\alpha_{i})=1_{Z_{\eta}}}\\
F_{\eta}(i)=F(i)\ar[dr]_{\rho_{i}^{F}}\ar[rr]^{\xi_{i}^{\eta}} &  & \kappa^{\Sigma}(Z_{\eta})(i)=Z_{\eta}\\
 & C_{F}\ar[ur]_{\mu_{\eta}}
}
\]
Indeed, such claim follows from the following equalities of morphisms.
\[
\begin{aligned}\mu_{\eta}\circ(\rho_{i}^{F}\circ\phi_{i}) & =\mu_{\eta}\circ(\col_{\Sigma}(\phi)\circ\rho_{i}^{\kappa^{\Sigma}(A)}) & (\text{by definition of \ensuremath{\col_{\Sigma}(\phi)}})\\
 & =(f_{\eta}\circ\nabla^{A})\circ\rho_{i}^{\kappa^{\Sigma}(A)} & \text{(By (1.1))}\\
 & =f_{\eta}\circ1_{A}=f_{\eta}=1_{Z_{\eta}}\circ f_{\eta}
\end{aligned}
\]

{The following result shows that the morphism $f_{\eta}$ coincide
with the left morphism in the sequence $\Theta(\eta)$ described in
the introduction (see \cite{argudin2021yoneda}) when $\Sigma$ is
a set.} 
\begin{lem}
Let $\eta:\;\kappa^{\Sigma}(A)\stackrel{\phi}{\hookrightarrow}F\stackrel{\psi}{\twoheadrightarrow}G$
be an exact sequence in $\Fun(\Sigma,\A)$, where $\Sigma$ is a small
category and $A$ is an object in $\A$, and let $Z_{\eta}$ be as
in the diagram (1.1). Then, the pair $(Z_{\eta},\xi^{\eta})$ is the
$\Sigma_{*}$-colimit of $F_{\eta}$. In particular, $Z_{\eta}=\col_{\Sigma_{*}}(F_{\eta})$. 
\end{lem}

\begin{proof}
Let $C$ be an object in $\A$ and let $\varphi:F_{\eta}\to\kappa^{\Sigma_{*}}(C)$
be a natural transformation in $\Fun(\Sigma_{*},\A)$. By definition,
we have that $\varphi_{j}\circ F(\lambda)=\varphi_{j}\circ F_{\eta}(\lambda)=\varphi_{i}$,
for every morphism $\lambda:i\to j$ in $\Sigma$, and the following
commutative diagram in $\A$, for all $i\in I$: 
\[
\xymatrix{F_{\eta}(*)=A\ar[r]^{\varphi_{*}}\ar[d]_{F_{\eta}(\alpha_{i})=\phi_{i}} & (\kappa^{\Sigma}(C))(*)=C\ar[d]^{(\kappa^{\Sigma}(C))(\alpha_{i})=1_{C}}\\
F_{\eta}(i)=F(i)\ar[r]^{\varphi_{i}} & (\kappa^{\Sigma}(C))(i)=C
}
\]
In particular, we get that the family of morphisms $(\varphi_{i})_{i\in Ob(\Sigma)}:=\varphi\restriction_{Ob(\Sigma)}$
is a natural transformation in $\Fun(\Sigma,\A)$. Now, from the universal
property of colimits, we obtain a unique morphism $\pi:C_{F}\to C$
in $\A$ such that the equality $\kappa_{\pi}^{\Sigma}\circ\rho^{F}=\varphi\restriction_{Ob(\Sigma)}$
holds in $\Fun(\Sigma,\A)$. In particular, we get 
\[
\begin{aligned}(\varphi_{*}\circ\nabla^{A})\circ\rho_{i}^{\kappa^{\Sigma}(A)} & =\varphi_{*}\circ1_{A}\\
 & =\varphi_{i}\circ\phi_{i}\\
 & =(\pi\circ\rho_{i}^{F})\circ\phi_{i}\\
 & =(\pi\circ\col_{\Sigma}(\phi))\circ\rho_{i}^{\kappa^{\Sigma}(A)}
\end{aligned}
\]
for all $i\in Ob(\Sigma).$ Using once again the universal property
of colimits we deduce that $\varphi_{*}\circ\nabla^{A}=\pi\circ\col_{\Sigma}(\phi)$.
Now, from the universal property of pushouts, we know that there is
a unique morphism $\sigma:Z_{\eta}\to C$ such that $\sigma\circ f_{\eta}=\varphi_{*}$
and $\sigma\circ\mu_{\eta}=\pi$. And hence, it follows that the natural
transformation $\kappa_{\sigma}^{\Sigma{_{*}}}$ satisfy the following
equation in the category $\Fun(\Sigma_{*},\A)$: 
\[
\kappa_{\sigma}^{\Sigma{_{*}}}\circ\xi^{\eta}=\varphi.
\]

It remains to prove the uniqueness of $\sigma$. Let $\sigma':Z_{\eta}\to C$
be a morphism in $\A$ such that $\kappa_{\sigma'}^{\Sigma{_{*}}}\circ\xi^{\eta}=\varphi$.
Then, $\varphi_{*}=(\kappa_{\sigma'}^{\Sigma{_{*}}})_{*}\circ\xi_{*}^{\eta}=\sigma'\circ f_{\eta}$
and $\varphi_{i}=(\kappa_{\sigma'}^{\Sigma{_{*}}})_{i}\circ\xi_{i}^{\eta}=\sigma'\circ(\mu_{\eta}\circ\rho_{i}^{F})=(\sigma'\circ\mu_{\eta})\circ\rho_{i}^{F}$,
for all $i\in Ob(\Sigma)$. Thus, $\kappa_{\sigma'\circ\mu_{\eta}}^{\Sigma}\circ\rho^{F}=\varphi\restriction_{Ob(\Sigma)}$.
And hence, from the universal property of colimits, $\sigma'\circ\mu_{\eta}=\pi$.
Finally, by the universal property of pushouts, we have $\sigma=\sigma'$
as desired. 
\end{proof}

\section{Main results}

In this section we will use the previous results to obtain a characterization
for the exactness of the functors $\colim(-):\Fun(\Sigma,\mathcal{A})\rightarrow\mathcal{A}$
and $\limite_{\Sigma}(-):\Fun(\Sigma,\mathcal{B})\rightarrow\mathcal{B}${,
when $\A$ and $\B$ are $\Sigma$-co-complete and $\Sigma$-complete
abelian categories, respectively, for all small category $\Sigma$}.


\begin{thm}
\label{teo:firstmain} {Let $\Sigma$ be a small category and let
$\A$ be a $\Sigma$-co-complete abelian category. Then, the following
conditions are equivalent:} 
\begin{enumerate}
\item the functor $\col_{\Sigma}:\Fun(\Sigma,\A)\to\A$ is exact; 
\item $f_{\eta}$ is a monomorphism, for every exact sequence $\eta:\;\kappa^{\Sigma}(A)\stackrel{\phi}{\hookrightarrow}F\stackrel{\psi}{\twoheadrightarrow}G$
in $\Fun(\Sigma,\A)$, where $A$ is an object in $\A$. 
\end{enumerate}
\end{thm}

\begin{proof}
(a) $\Rightarrow$ (b). It is clear by properties of the pushouts
(see (1.1)).

(b) $\Rightarrow$ (a). Let $H\stackrel{\zeta}{\hookrightarrow}N\stackrel{\chi}{\twoheadrightarrow}G$
be an exact sequence in $\Fun(\Sigma,\A)$. Consider the following
commutative diagram in $\Fun(\Sigma,A)$ with exact rows, where the
left square is the pushout of $\zeta$ and $\rho^{H}$ 
\[
\xymatrix{0\ar[r] & H\ar[r]^{\zeta}\ar[d]_{\rho^{H}} & N\ar[r]^{\chi}\ar[d]^{\omega} & G\ar[r]\ar@{=}[d] & 0\\
0\ar[r] & \kappa^{\Sigma}(C_{H})\ar[r]^{\phi} & F\ar[r]^{\psi}\pullbackcorner & G\ar[r] & 0
}
\]
From the universal property of pushouts, we obtain a unique natural
transformation $\upsilon:F\to\kappa^{\Sigma}(C_{N})$ such that $\upsilon\circ\omega=\rho^{N}$
and $\upsilon\circ\phi=\kappa_{\col_{\Sigma}(\zeta)}^{\Sigma}$ since
$\rho^{N}\circ\zeta=\kappa_{\col_{\Sigma}(\zeta)}^{\Sigma}\circ\rho^{H}$.
Note that there is a unique morphism $\gamma:C_{F}\to C_{N}$ such
that $\kappa_{\gamma}^{\Sigma}\circ\rho^{F}=\upsilon$ by the universal
property of colimits.

Now, set $\eta:=\kappa^{\Sigma}(C_{H})\stackrel{\phi}{\hookrightarrow}F\stackrel{\psi}{\twoheadrightarrow}G$
and, for each $i\in Ob(\Sigma)$, consider the following commutative
diagram in $\A$ with exact rows, where the upper and lower left squares
are pushouts (here $f_{\eta}$ is a monomorphism by the assumption
(b)): 
\[
\xymatrix{0\ar[r] & H(i)\ar[rr]^{\zeta_{i}}\ar[d]_{\rho_{i}^{H}} &  & N(i)\ar[d]_{\omega_{i}}\ar[rr]^{\chi_{i}} &  & G(i)\ar[r]\ar@{=}[d] & 0\\
0\ar[r] & C_{H}\ar[d]_{\rho_{i}^{\kappa^{\Sigma}(C_{H})}}\ar[rr]^{\phi_{i}} &  & F(i)\ar[rr]^{\psi_{i}}\ar[d]_{\rho_{i}^{F}}\pullbackcorner &  & G(i)\ar[r]\ar[d]_{\rho_{i}^{G}} & 0\\
 & C_{\kappa^{\Sigma}(C_{H})}\ar[d]_{\nabla^{C_{H}}}\ar[rr]^{\col_{\Sigma}(\phi)} &  & C_{F}\ar[d]_{\mu_{\eta}}\ar[rr]^{\col_{\Sigma}(\psi)} &  & C_{G}\ar[r]\ar@{=}[d] & 0\\
0\ar[r] & C_{H}\ar[rr]^{f_{\eta}} &  & Z_{\eta}\ar[rr]^{g_{\eta}}\pullbackcorner &  & C_{G}\ar[r] & 0
}
\]
In particular, we have the following equalities

\[
\begin{aligned}(\gamma\circ\col_{\Sigma}(\phi))\circ\rho_{i}^{\kappa^{\Sigma}(C_{H})} & =\gamma\circ(\rho_{i}^{F}\circ\phi_{i})\\
 & =\upsilon_{i}\circ\phi_{i}\\
 & =\col_{\Sigma}(\zeta)\\
 & =\col_{\Sigma}(\zeta)\circ1_{C_{H}}\\
 & =(\col_{\Sigma}(\zeta)\circ\nabla^{C_{H}})\circ\rho_{i}^{\kappa^{\Sigma}(C_{H})}
\end{aligned}
\]
for all $i\in Ob(\Sigma)$. So that $\kappa_{\gamma\circ\col_{\Sigma}(\phi)}^{\Sigma}\circ\rho^{\kappa^{\Sigma}(C_{H})}=\kappa_{\col_{\Sigma}(\zeta)\circ\nabla^{C_{H}}}^{\Sigma}\circ\rho^{\kappa^{\Sigma}(C_{H})}$,
and hence $\col_{\Sigma}(\zeta)\circ\nabla^{C_{H}}=\gamma\circ\col_{\Sigma}(\phi)$.
Therefore, there is a unique morphism $\theta:Z_{\eta}\to C_{N}$
such that $\theta\circ\mu_{\eta}=\gamma$ and $\theta\circ f_{\eta}=\col_{\Sigma}(\zeta)$.
We claim that the pair $(Z_{\eta},\kappa_{\mu_{\eta}}^{\Sigma}\circ(\rho^{F}\circ\omega):N\to\kappa^{\Sigma}(Z_{\eta}))$
is a colimit of $N$. Indeed, by the universal property of colimits,
there is a unique morphism $m:C_{N}\to Z_{\eta}$ in $\A$ such that
$\kappa_{m}^{\Sigma}\circ\rho^{N}=\kappa_{\mu_{\eta}}^{\Sigma}\circ(\rho^{F}\circ\omega)$.
Therefore, the claim holds once we show that $m$ is an isomorphism
in $\A$. For this, observe that, on the one hand: 
\[
\begin{aligned}(\theta\circ m)\circ\rho_{i}^{N} & =\theta\circ(\mu_{\eta}\circ(\rho_{i}^{F}\circ\omega_{i}))\\
 & =\gamma\circ(\rho_{i}^{F}\circ\omega_{i})\\
 & =\upsilon_{i}\circ\omega_{i}\\
 & =\rho_{i}^{N}=1_{C_{N}}\circ\rho_{i}^{N}
\end{aligned}
\]
for all $i\in Ob(\Sigma)$. And hence, $\theta\circ m=1_{C_{N}}$
by the universal property of colimits. On the other hand, from the
above diagram we have that 
\[
f_{\eta}=f_{\eta}\circ1_{C_{H}}=f_{\eta}\circ\nabla^{C_{H}}\circ\rho_{i}^{\kappa^{\Sigma}(C_{H})}=\mu_{\eta}\circ\rho_{i}^{F}\circ\phi_{i}\mbox{,}
\]
for all $i\in Ob(\Sigma)$. In particular, for every $i\in Ob(\Sigma)$,
we get 
\[
f_{\eta}\circ\rho_{i}^{H}=\mu_{\eta}\circ\rho_{i}^{F}\circ\phi_{i}\circ\rho_{i}^{H}=\mu_{\eta}\circ\rho_{i}^{F}\circ\omega_{i}\circ\zeta_{i}=m\circ\rho_{i}^{N}\circ\zeta_{i}=m\circ\col_{\Sigma}(\zeta)\circ\rho_{i}^{H}.
\]
Therefore, $f_{\eta}=m\circ\col_{\Sigma}(\zeta)=\mu_{\eta}\circ\rho_{i}^{F}\circ\phi_{i}$
for all $i\in Ob(\Sigma)$. Now, for each $i\in Ob(\Sigma)$, we get
the following commutative diagram in $\A$ (recall that $\mu_{\eta}\circ(\rho_{i}^{F}\circ\omega_{i})=m\circ\rho_{i}^{N}$):
\[
\xymatrix{ & N(i)\ar[d]^{\omega_{i}}\ar@(u,r)[ddr]^{m\circ\rho_{i}^{N}}\\
C_{H}\ar[r]^{\phi_{i}}\ar@(d,d)[drr]_{m\circ\col_{\Sigma}(\zeta)} & F(i)\ar[dr]_{\mu_{\eta}\circ\rho_{i}^{F}}\\
 &  & Z_{\eta}
}
\]
From the universal property of the upper pushout in the diagram, we
obtain that $\mu_{\eta}\circ\rho_{i}^{F}=m\circ\upsilon_{i}$, for
all $i\in Ob(\Sigma)$. But, each $\upsilon_{i}$ coincides with the
composition $\gamma\circ\rho_{i}^{F}$, and hence $\mu_{\eta}=m\circ\gamma$.
Now, notice that the composition $m\circ\theta$ and $1_{Z_{\eta}}$
complete the following diagram in $\A$ and, therefore, such morphisms
coincide and our claim holds (since $m\circ\theta=1_{Z_{\eta}}$ and
$\theta\circ m=1_{C_{N}}$). 
\[
\xymatrix{ & C_{F}\ar[d]_{\mu_{\eta}}\ar@(r,u)[ddr]^{m\circ\gamma}\\
C_{H}\ar[r]^{f_{\eta}}\ar@(r,l)[drr]_{m\circ\col_{\Sigma}(\zeta)} & Z_{\eta}\ar@{--}[dr]\\
 &  & Z_{\eta}
}
\]
Finally, from the equality $\theta\circ f_{\eta}=\col_{\Sigma}(\zeta)$
we get that the morphism $\col_{\Sigma}(\zeta)$ is a monomorphism
as desired. 
\end{proof}
\begin{rem}
\label{rem:defPsi} {In the setting of the previous theorem, we get
that for each $F\in\Fun(\Sigma,\A)$ and $A\in Ob(\A)$ there is a
natural map } 
\[
\Psi_{F,A}^{\Sigma}:\text{Ext}_{\A}^{1}(C_{F},A)\to\text{Ext}_{\Fun(\Sigma,\A)}^{1}(F,\kappa^{\Sigma}(A))
\]
both in $A$ and in $F$, defined as $\Psi_{F,A}^{\Sigma}(\overline{\epsilon}):=\overline{\kappa^{\Sigma}(\epsilon)\cdot\rho^{F}}$,
where $\kappa^{\Sigma}(\epsilon)\cdot\rho^{F}$ denote the upper exact
sequence in the following commutative diagram in $\Fun(\Sigma,\A)$
with exact rows (here the right square is a pullback of the respective
morphisms): 
\[
\xymatrix{\hspace{1cm}0\ar[r] & \kappa^{\Sigma}(A)\ar[r]\ar@{=}[d] & G_{\kappa^{\Sigma}(\epsilon)}\pushoutcorner\ar[r]\ar[d] & F\ar[r]\ar[d]^{\rho^{F}} & 0\\
\kappa^{\Sigma}(\epsilon):\;0\ar[r] & \kappa^{\Sigma}(A)\ar[r] & \kappa^{\Sigma}(B)\ar[r] & \kappa^{\Sigma}(C_{F})\ar[r] & 0
}
\]
\end{rem}

\begin{lem}
\label{lem:injPsi} {Let $\Sigma$ be a small category and let $\A$
be a $\Sigma$-co-complete abelian category. Then, the map $\Psi_{F,A}^{\Sigma}$
is injective, for all $F\in\Fun(\Sigma,\A)$ and $A\in Ob(\A)$.}
\end{lem}


\begin{proof}
Let $\epsilon:\;A\stackrel{f}{\hookrightarrow}B\stackrel{g}{\twoheadrightarrow}C_{F}$
be an extension in $\A$ such that $\Psi_{F,A}^{\Sigma}(\epsilon)$
is the trivial extension in $\text{Ext}_{\Fun(\Sigma,\A)}^{1}(\kappa^{\Sigma}(A),F)$.
Thus, there is a natural transformation $\tau:F\to\kappa^{\Sigma}(B)$
such that $\rho^{F}=\kappa^{\Sigma}(g)\circ\tau$. Now, from the universal
property of colimits, we obtain a unique morphism $h:C_{F}\to B$
such that $\tau=\kappa_{h}^{\Sigma}\circ\rho^{F}$. Using once again
the universal property of colimits, we deduce that $g\circ h=1_{C_{F}}$
and hence $\overline{\epsilon}$ is the trivial extension in $\text{Ext}_{\A}^{1}(C_{F},A)$. 
\end{proof}
The following is the second main result of the paper. 
\begin{thm}
\label{teo:secondmain} {Let $\Sigma$ be a small category and let
$\A$ be a $\Sigma$-co-complete abelian category. Then, the following
conditions are equivalent:} 
\begin{enumerate}
\item the functor $\col_{\Sigma}:\Fun(\Sigma,\A)\to\A$ is exact; 
\item the map $\Psi_{F,A}^{\Sigma}$ is bijective, for every $F\in\Fun(\Sigma,\A)$
and $A\in Ob(\A)$; 
\item the bifunctors $\text{Ext}_{\A}^{1}(\col_{\Sigma}(?),-)$ and $\text{Ext}_{\Fun(\Sigma,\A)}^{1}(?,\kappa^{\Sigma}(-))$
are naturally isomorphic in the obvious way. 
\end{enumerate}
\end{thm}

\begin{proof}
(b) $\Leftrightarrow$ (c) Is clear by definition of the maps $\Psi_{F,A}^{\Sigma}$
(see Remark \ref{rem:defPsi}).

(a) $\Rightarrow$ (b) Let $F$ be a functor in $\Fun(\Sigma,\A)$
and let $A$ be an object in $\A$. By Lemma \ref{lem:injPsi} our
task is reduced to check that $\Psi_{F,A}^{\Sigma}$ is surjective.
For this, consider an arbitrary exact sequence $\eta:\;\kappa^{\Sigma}(A)\stackrel{\phi}{\hookrightarrow}G\stackrel{\psi}{\twoheadrightarrow}F$
in $\Fun(\Sigma,\A)$. Now, using the exactness of the functor $\col_{\Sigma}$,
we obtain the following commutative diagram in $\A$ with exact rows,
where the left square is a pushout (see (1.1)). 
\[
\xymatrix{\hspace{0,8cm}0\ar[r] & C_{\kappa^{\Sigma}(A)}\ar[rr]^{\col_{\Sigma}(\phi)}\ar[d]_{\nabla^{A}} &  & C_{G}\ar[rr]^{\col_{\Sigma}(\psi)}\ar[d]_{\mu_{\eta}} &  & C_{F}\ar[r]\ar@{=}[d] & 0\\
\epsilon:\;0\ar[r] & A\ar[rr]^{f_{\eta}} &  & Z_{\eta}\pullbackcorner\ar[rr]^{g_{\eta}} &  & C_{F}\ar[r] & 0
}
\]
Let $\epsilon$ be the lower exact sequence in the above diagram.
We claim that $\Psi_{F,A}^{\Sigma}(\overline{\epsilon})=\overline{\eta}$.
To show this, notice that we have the following commutative diagram
in $\Fun(\Sigma,\A)$:

\[
\xymatrix{\eta:\;0\ar[r] & \kappa^{\Sigma}(A)\ar[rr]^{\phi}\ar[d]_{\rho^{\kappa^{\Sigma}(A)}} &  & G\ar[rr]^{\psi}\ar[d]_{\rho^{G}} &  & F\ar[r]\ar[d]_{\rho^{F}} & 0\\
0\ar[r] & \kappa^{\Sigma}(C_{\kappa^{\Sigma}(A)})\ar[rr]^{\kappa^{\Sigma}(\col_{\Sigma}(\phi))}\ar[d]_{\kappa_{\nabla^{A}}^{\Sigma}} &  & \kappa^{\Sigma}(C_{G})\ar[rr]^{\kappa^{\Sigma}(\col_{\Sigma}(\psi))}\ar[d]_{\kappa_{\mu_{\eta}}^{\Sigma}} &  & \kappa^{\Sigma}(C_{F})\ar@{=}[d]\ar[r] & 0\\
\kappa^{\Sigma}(\epsilon):\;0\ar[r] & \kappa^{\Sigma}(A)\ar[rr]^{\kappa^{\Sigma}(f_{\eta})} &  & \kappa^{\Sigma}(Z_{\eta})\ar[rr]^{\kappa^{\Sigma}(g_{\eta})} &  & \kappa^{\Sigma}(C_{F})\ar[r] & 0
}
\]

Using the fact that $\kappa_{\nabla^{A}}^{\Sigma}\circ\rho^{\kappa^{\Sigma}(A)}=1_{\kappa^{\Sigma}(A)}$,
we obtain that the right outer rectangle in the above diagram is a
pullback (see \cite[dual of Lemma 5.2, p.35]{Popescu}). And, hence
$\Psi_{F,A}^{\Sigma}(\overline{\epsilon})=\overline{\eta}$.

(b) $\Rightarrow$ (a) By Theorem \ref{teo:firstmain}, it is enough
to check that $f_{\eta}$ is a monomorphism, for every exact sequence
$\eta:\;\kappa^{\Sigma}(A)\stackrel{\phi}{\hookrightarrow}G\stackrel{\psi}{\twoheadrightarrow}F$
in $\Fun(\Sigma,\A)$, where $A$ is an object in $\A$. Observe that,
by hypothesis, there exists an exact sequence $\epsilon:\;A\stackrel{f}{\hookrightarrow}B\stackrel{g}{\twoheadrightarrow}C_{F}$
in $\mathcal{A}$ such that $\Psi_{F,A}^{\Sigma}(\overline{\epsilon})=\overline{\eta}$.
So that we have the following commutative diagram in $\Fun(\Sigma,\A)$.
\[
\xymatrix{\Psi_{F,A}^{\Sigma}(\epsilon):\;0\ar[r] & \kappa^{\Sigma}(A)\ar[r]^{\phi}\ar@{=}[d] & G\pushoutcorner\ar[r]^{\psi}\ar[d]^{u} & F\ar[r]\ar[d]^{\rho^{F}} & 0\\
\kappa^{\Sigma}(\epsilon):\;0\ar[r] & \kappa^{\Sigma}(A)\ar[r]^{\kappa^{\Sigma}(f)} & \kappa^{\Sigma}(B)\ar[r]^{\kappa^{\Sigma}(g)} & \kappa^{\Sigma}(C_{F})\ar[r] & 0
}
\]
Now, from the universal property of colimits, there is a unique morphism
$h:C_{G}\to B$ such that $\kappa_{h}^{\Sigma}\circ\rho^{G}=u$. Using
the fact that $\kappa^{\Sigma}(f)=u\circ\phi$ and $\kappa^{\Sigma}(\col_{\Sigma}(\phi))\circ\rho^{\kappa^{\Sigma}(A)}=\rho^{G}\circ\phi$,
we deduce that 
\[
\begin{aligned}\kappa^{\Sigma}(h\circ\col_{\Sigma}(\phi))\circ\rho^{\kappa^{\Sigma}(A)} & =\kappa_{h}^{\Sigma}\circ\kappa^{\Sigma}(\col_{\Sigma}(\phi))\circ\rho^{\kappa^{\Sigma}(A)}\\
 & =\kappa_{h}^{\Sigma}\circ\rho^{G}\circ\phi\\
 & =u\circ\phi\\
 & =\kappa^{\Sigma}(f)\\
 & =\kappa^{\Sigma}(f)\circ1_{\kappa^{\Sigma}(A)}\\
 & =\kappa^{\Sigma}(f)\circ\kappa_{\nabla^{A}}^{\Sigma}\circ\rho^{\kappa^{\Sigma}(A)}\\
 & =\kappa^{\Sigma}(f\circ\nabla^{A})\circ\rho^{\kappa^{\Sigma}(A)}.
\end{aligned}
\]
Then, by the universal property of colimits, $h\circ\col_{\Sigma}(\phi)=f\circ\nabla^{A}$.
Now, from the universal property of pushouts, we get a unique morphism
$\gamma:Z_{\eta}\to B$ such that $\gamma\circ f_{\eta}=f$ (see (1.1)).
Hence, $f_{\eta}$ is a monomorphism. 
\end{proof}
\begin{rem}
Let $\mathcal{X}$ and $\mathcal{Y}$ be abelian categories, and $F:\mathcal{X}\rightarrow\mathcal{Y}$
and $G:\mathcal{Y}\rightarrow\mathcal{X}$ be functors such that $(F,G)$
is an adjoint pair. We will say that $(F,G)$ is an \textbf{Ext-adjoint
pair} if there is a natural isomorphism $\Ext_{\mathcal{X}}^{1}(-,G(?))\cong\Ext_{\mathcal{Y}}^{1}(F(-),?)$.
Note that the statement (c) of the above theorem can be restated by
saying that $(\operatorname{colim}_{\Sigma},\kappa^{\Sigma})$ is
an Ext-adjoint pair. It should be noted that not every adjoint pair
is Ext-adjoint. The reader can find counterexamples of this in \cite[pp.29--30]{adams1967adjoint}. 
\end{rem}

By duality we get the following straightforward corollary. 
\begin{cor}
\label{cor:lim-exac} {Let $\Sigma$ be a small category and let
$\B$ be a $\Sigma$-complete abelian category. Then, the following
conditions are equivalent:} 
\begin{enumerate}
\item the functor $\lim_{\Sigma}:\Fun(\Sigma,\B)\to\B$ is exact; 
\item the maps $\Phi_{B,F}^{\Sigma}:\text{Ext}_{\B}^{1}(B,{L_{F}})\to\text{Ext}_{\Fun(\Sigma,\B)}^{1}(\kappa^{\Sigma}(B),F)$
defined as $\Phi_{B,F}^{\Sigma}(\overline{\epsilon})=\overline{\varrho^{F}\cdot\kappa^{\Sigma}(\epsilon)}$,
where $\varrho^{F}\circ\kappa^{\Sigma}(\epsilon)$ denotes the lower
exact sequence in the below commutative diagram in $\Fun(\Sigma,\B)$
with exact rows (here the left square is a pushout), are natural isomorphisms
both in $B$ and in $F$. 
\[
\xymatrix{\kappa^{\Sigma}(\epsilon):\;0\ar[r] & \kappa^{\Sigma}(L_{F})\ar[r]\ar[d]_{\varrho^{F}} & \kappa^{\Sigma}(A)\ar[r]\ar[d] & \kappa^{\Sigma}(B)\ar[r]\ar@{=}[d] & 0\\
0\ar[r] & F\ar[r] & G\pullbackcorner\ar[r] & \kappa^{\Sigma}(B)\ar[r] & 0
}
\]
\item the bifunctors $\text{Ext}_{\B}^{1}(?,\lim_{\Sigma}(-))$ and $\text{Ext}_{\Fun(\Sigma,\B)}^{1}(\kappa^{\Sigma}(?),-)$
are naturally isomorphic in the obvious way. 
\end{enumerate}
\end{cor}

Now, we give a new proof of the following result, which is a particular
case of the dual of \cite[Theorem 3.2.8]{Popescu}. 
\begin{cor}
\label{cor:kappaSigmapreserve} {In the setting of the previous corollary.
If $\B$ has enough projectives, then the functor $\lim_{\Sigma}$
is exact if and only if the functor $\kappa^{\Sigma}$ }preserves
projective objects. 
 
\end{cor}

\begin{proof}
($\Rightarrow$) It follows by Corollary \ref{cor:lim-exac}(c).

($\Leftarrow$) Let $B$ be an object in $\B$ and let $F$ be a functor
in $\Fun(\Sigma,\B)$. Consider an exact sequence in $\B$ of the
form: 
\[
\xymatrix{\epsilon:\;0\ar[r] & A\ar[r] & P\ar[r] & B\ar[r] & 0}
\mbox{,}
\]
where $P$ is a projective object of $\B$. Applying the functors
$\Gamma:=\text{Hom}_{\B}(-,\limite_{\Sigma}(F))$ and $\Upsilon:=\text{Hom}_{\Fun(\Sigma,\B)}(-,F)$
on $\epsilon$ and $\kappa^{{\Sigma}}(\epsilon)$, respectively, we
get the commutative diagram with exact rows 
\[
\xymatrix{0\ar[r] & \Gamma(B)\ar[r]\ar[d]_{f_{1}} & \Gamma(P)\ar[r]\ar[d]_{f_{2}} & \Gamma(A)\ar[r]\ar[d]_{f_{3}} & \Gamma^{1}(B)\ar[r]\ar[d]_{\Phi_{B,F}^{\Sigma}} & \Gamma^{1}(P)=0\\
0\ar[r] & \Upsilon(\kappa^{{\Sigma}}(B))\ar[r] & \Upsilon(\kappa^{{\Sigma}}(P))\ar[r] & \Upsilon(\kappa^{{\Sigma}}(A))\ar[r] & \Upsilon^{1}(\kappa^{{\Sigma}}(B))\ar[r] & \Upsilon^{1}(\kappa^{{\Sigma}}(P))
}
\]
where $\Gamma^{1}$ and $\Upsilon^{1}$ denote the functors $\text{Ext}_{\B}^{1}(-,\limite_{\Sigma}(F))$
and $\text{Ext}_{\Fun(\Sigma,\B)}^{1}(-,F)$ respectively, and the
morphisms $f_{i}$'s are the natural isomorphisms given by the adjoint
pair $(\kappa^{\Sigma},\lim_{\Sigma}).$ Using the hypothesis on $\kappa^{\Sigma}$,
we know that $\Upsilon^{1}(\kappa^{{\Sigma}}(P))=0$ and hence $\Phi_{B,F}^{\Sigma}$
is an isomorphism by the Five Lemma. Thus, this implication follows
by Corollary \ref{cor:lim-exac}(b). 
\end{proof}
The following result is a key for a direct consequence of the above
corollary. 

\begin{lem}
\label{lem:kappasetinjproj} Let $\A$ be an abelian category. Then,
the functor $\kappa^{\Sigma}$ preserve injective and projective objects,
for every set $\Sigma$ viewed as a small category. 
\end{lem}

\begin{proof}
We show that the functor $\kappa^{\Sigma}$ preserves projective objects.
The rest of the statement can be proved by similar arguments. Let
$P$ be a projective object in $\A$ and let $\epsilon:\;F\stackrel{\psi}{\hookrightarrow}G\stackrel{\phi}{\twoheadrightarrow}\kappa^{\Sigma}(P)$
be an exact sequence in $\Fun(\Sigma,\A)$. In such case, $\phi_{i}$
is a split epimorphism, for all $i\in Ob(\Sigma)=\Sigma$. Thus, for
each $i\in\Sigma$, we can take a morphism $\xi_{i}:\kappa^{\Sigma}(P)(i)=P\to G(i)$
such that $\phi_{i}\circ\xi_{i}=1_{\kappa^{\Sigma}(P)(i)}$. Now,
using the fact that the only morphisms in $\Sigma$ are the identity
morphisms, we obtain that the family of morphisms $\xi:=(\xi_{i})_{i\in\Sigma}$
is a natural transformation from the functor $\kappa^{\Sigma}(P)$
to the functor $G$. Finally, note that $\xi\circ\phi=1_{\kappa^{\Sigma}(P)}$
and hence $\overline{\epsilon}$ is the trivial extension in $\text{Ext}_{\Fun(\Sigma,\A)}^{1}(\kappa^{\Sigma}(P),F)$.
Therefore, $\kappa^{\Sigma}(P)$ is a projective object in $\Fun(\Sigma,\A)$
as desired. 
\end{proof}
From Corollary \ref{cor:kappaSigmapreserve} and its dual, together
with the Lemma \ref{lem:kappasetinjproj}, we deduce the Corollary
3.2.9 in \cite{Popescu}. 
\begin{cor}
\label{cor:ab3 con suf inj es ab4} Let $\A$ be an Ab3 abelian category
and let $\B$ be an Ab3{*} abelian category. Then, the following assertions
hold: 
\begin{enumerate}
\item $\A$ is Ab4, whenever $\A$ has enough injectives; 
\item $\B$ is Ab4{*}, whenever $\A$ has enough projectives. 
\end{enumerate}
\end{cor}

The following result is straightforward. 
\begin{lem}
\label{lem:XiTheta} Let $\A$ be an abelian category and let $\Sigma$
be a set. Given an exact sequence $\epsilon:\;H\stackrel{\phi}{\hookrightarrow}G\stackrel{\psi}{\twoheadrightarrow}F$
in $\Fun(\Sigma,\A)$, consider the exact sequence $\epsilon_{i}:\;H(i)\stackrel{\phi_{i}}{\hookrightarrow}G(i)\stackrel{\psi_{i}}{\twoheadrightarrow}F(i)$
$\forall i\in\Sigma$. Then, the assignment $\epsilon\mapsto\prod_{i\in\Sigma}\epsilon_{i}$
gives rise a family of isomorphisms 
\[
\xymatrix{\Xi_{{F,}A}^{{\Sigma}}:\text{Ext}_{\Fun(\Sigma,\A)}^{1}(F,\kappa^{\Sigma}(A))\ar[r] & \prod_{i\in\Sigma}\text{Ext}_{\A}^{1}(F(i),A)\\
\Theta_{{F,}A}^{{\Sigma}}:\text{Ext}_{\Fun(\Sigma,\A)}^{1}(\kappa^{\Sigma}(A),F)\ar[r] & \prod_{i\in\Sigma}\text{Ext}_{\A}^{1}(A,F(i))
}
\]
for all $F$ in $\Fun(\Sigma,\A)$ and $A$ in $Ob(\A)$. 
\end{lem}

Let $\A$ be an Ab3 abelian category and $\Sigma$ be a set viewed
as a small category. In this context, $\col_{\Sigma}(F)=\coprod_{i\in\Sigma}F(i)$
for all $F$ in $\Fun(\Sigma,\A)$. We can use this to get as a corollary
\cite[Theorems 4.9 and 4.11]{argudin2021yoneda}. 
\begin{cor}
Let $\A$ be an Ab3 abelian category and let $\B$ be an Ab3{*} abelian
category. Then, the following assertions hold: 
\begin{enumerate}
\item $\A$ is Ab4 if, and only if, the canonical map 
\[
\text{Ext}_{\A}^{1}(\coprod_{i\in I}A_{i},A)\to\prod_{i\in I}\text{Ext}_{\A}^{1}(A_{i},A)
\]
is an isomorphism, for all $A$ and $(A_{i})_{i\in I}$ object and
family of objects in $\A$, respectively. 
\item $\B$ is Ab4{*} if, and only if, the canonical map 
\[
\text{Ext}_{\B}^{1}(B,\prod_{i\in I}B_{i})\to\prod_{i\in I}\text{Ext}_{\B}^{1}(B,B_{i})
\]
is an isomorphism, for all $B$ and $(B_{i})_{i\in I}$ object and
family of objects in $\B$, respectively. 
\end{enumerate}
\end{cor}

\begin{proof}
We prove statement (a), statement (b) follows by duality. Let $I$
be a set and let $(A_{i})_{i\in I}$ be a family of objects in $\A$.
Then, the assignment $i\mapsto A_{i}$ give rise to a functor $F:I\to\A$.
Notice that for each $A$ in $Ob(A)$, we have that the map in (a)
coincide with the compositions $\Xi_{{F,}A}^{{I}}\circ\Psi_{{F,}A}^{{I}}$
(see Theorem \ref{teo:secondmain} and Lemma \ref{lem:XiTheta}).
On the other hand, every functor in $\Fun(I,\A)$ can be viewed as
an $I$-indexed family of objects in $\A$. Then, the result follows
by Theorem \ref{teo:secondmain} and Lemma \ref{lem:XiTheta}. 
\end{proof}

\subsection{Projective effacements and the Ab4{*} condition}

Let $\mathcal{A}$ be an abelian category. We say that an epimorphism
$\pi:B\to A$ in $\mathcal{A}$ is a \textbf{projective effacement}
if $\Ext_{\mathcal{A}}^{1}(\pi,X)=0$ for all $X\in\mathcal{A}$.
This means that, for every epimorphism $p:E\to A$, there is a morphism
$b:B\to E$ such that $p\circ b=\pi$. In case every object $A\in\mathcal{A}$
admits a projective effacement $B\to A$, we say that $\mathcal{A}$
has \textbf{enough} \textbf{projective effacements}. 

In the context of Grothendieck categories, the existence of enough
projective effacements is used to give a characterization of the Ab4{*}
condition (see \cite[Thm. A]{B} or \cite[Cor. 1.4]{R}). The goal
of this section is to prove that the existence of enough projective
effacements implies the Ab4{*} condition for any Ab3{*} abelian category.
\begin{prop}
Let $\mathcal{A}$ be an Ab3{*} abelian category. If $\mathcal{A}$
has enough projective effacements, then $\mathcal{A}$ is Ab4{*}. 
\end{prop}

\begin{proof}
Consider a set $\Sigma$, a functor $F\in\Fun(\Sigma,\mathcal{A})$,
and a short exact sequence $F\stackrel{\alpha}{\hookrightarrow}E\stackrel{\beta}{\twoheadrightarrow}\kappa^{\Sigma}A$
in $\Fun(\Sigma,\mathcal{A})$. In other words, we have a family of
short exact sequences $\{F(i)\stackrel{\alpha_{i}}{\hookrightarrow}E(i)\stackrel{\beta_{i}}{\twoheadrightarrow}A\}_{i\in\Sigma}.$
Consider the following pull-back diagram.
\[
\xymatrix{0\ar[r] & \prod_{i\in\Sigma}F(i)\ar[r]\ar@{=}[d] & G\pushoutcorner\ar[rr]^{f_{\eta}}\ar[d] &  & A\ar[d]^{\Delta^{A}}\\
0\ar[r] & \prod_{i\in\Sigma}F(i)\ar[r]^{\prod_{i\in\Sigma}\alpha_{i}} & \prod_{i\in\Sigma}E(i)\ar[rr]^{\prod_{i\in\Sigma}\beta_{i}} &  & A^{\Sigma}
}
\]
It is enough to prove that $f_{\eta}$ always is an epimorphism to
conclude that $\mathcal{A}$ is Ab4{*} by the dual of Theorem \ref{teo:firstmain}.
For this, consider a projective effacement $\pi:B\to A$. Observe
that, for every $i\in\Sigma$, there is $\gamma_{i}:B\to E(i)$ such
that $\beta_{i}\circ\gamma_{i}=\pi$; and that $\Delta^{A}\circ\pi=\pi^{\Sigma}\circ\Delta^{B}$.
And thus, since $(\prod_{i\in\Sigma}\beta_{i})\circ(\prod_{i\in\Sigma}\gamma_{i})\circ\Delta^{B}=\pi^{\Sigma}\circ\Delta^{B}=\Delta^{A}\circ\pi$,
it follows from the universal property of pull-backs that there is
a morphism $x:B\to G$ such that $f_{\eta}\circ x=\pi$. Therefore,
we have that $f_{\eta}$ is an epimorphism since $\pi$ is epic. 

\[
\xymatrix{ & B\ar@(r,u)[drrr]^{\pi}\ar[d]_{\Delta^{B}}\ar@{-->}[dr]^{x}\\
 & B^{\Sigma}\ar@(d,l)[dr]_{\prod_{i\in\Sigma}\gamma_{i}} & G\pushoutcorner\ar[rr]^{f_{\eta}}\ar[d] &  & A\ar[d]^{\Delta^{A}}\\
 &  & \prod_{i\in\Sigma}E(i)\ar[rr]^{\prod_{i\in\Sigma}\beta_{i}} &  & A^{\Sigma}
}
\]
\end{proof}

\subsection{Direct limits in the heart of a Happel-Reiten-Smal{\o} $t$-structure }

In this section we will be considering an Ab5 and Ab3{*} abelian category
$\mathcal{A}$ (e.g. a Grothendieck category), together with a torsion
pair $\mathfrak{t}=(\mathcal{T},\mathcal{F})$ in $\mathcal{A}$.
Recall that this means that $\Hom_{\mathcal{A}}(\mathcal{T},\mathcal{F})=0$,
and that every $A\in\mathcal{A}$ admits a short exact sequence $tA\stackrel{}{\hookrightarrow}A\stackrel{}{\twoheadrightarrow}(1:t)A$
in $\mathcal{A}$, where $tA\in\mathcal{A}$ and $(1:t)A\in\mathcal{F}$.
It is well-known that $\mathcal{F}$ is closed under subobjects, extensions
and products \cite[Prop. 2.2]{PS5}. Moreover, since $\mathcal{A}$
is Ab5, we have that $\mathcal{F}$ is closed under coproducts also
(see \cite[Thm. 3.1.9 and Cor. 3.1.3]{mitchell}).

In this setting, the torsion pair $\t$ induces a new abelian category
$\mathcal{H}_{\t}$ (namely, the heart of the Happel-Reiten-Smal{\o}
$t$-structure associated to $\mathfrak{t}$). Such abelian category
is described as the full subcategory of the objects $D$ in the derived
category $\mathcal{D}=\mathcal{D}(\mathcal{A})$ that admit a distinguished
triangle $F_{D}[1]\to D\to T_{D}[0]$ in $\mathcal{D}$, where $T_{D}\in\mathcal{T}$
and $F_{D}\in\mathcal{F}$. In fact, $\overline{\t}=(\mathcal{F}[1],\mathcal{T}[0])$
turns out to be a torsion pair in $\mathcal{H}_{\t}$ with torsion
radical $\mathcal{H}_{\t}\to\mathcal{A}$, $D\mapsto H^{-1}(D)[1]=F_{D}[1]$.
We will be assuming that $\t$ is a torsion pair such that $\mathcal{H}_{\t}$
has Hom sets (see \cite[Qn. 2.11 and Prop. 2.12]{PS5}). It is worth
to point out that, since $\mathcal{A}$ is Ab4, $\mathcal{D}$ has
coproducts (see \cite[Lem. 4.1.5]{K}). In fact, for a family of objects
$\{A_{i}\}_{i\in I}$ in $\mathcal{A}$, the coproduct $\coprod_{i\in I}(A_{i}[0])$
in $\mathcal{D}$ is equal to $(\coprod_{i\in I}A_{i})[0]$. This
implies that, for a family $\{H_{i}\}_{i\in I}\subseteq\mathcal{H}_{\t}$,
the coproduct $\coprod_{i\in I}H_{i}$ in $\mathcal{D}$ is also a
coproduct in $\mathcal{H}_{\t}$. In particular, $\mathcal{H}_{\t}$
is Ab3. The reader is referred to \cite{HRS,PS1,PS5} for unexplained
notation and terminology. 

In \cite[Thm. 4.8]{PS1} and \cite[Thm. 4.18]{PS5}, it was proved
that $\mathcal{H}_{\mathfrak{t}}$ is Ab5 if $\mathcal{F}$ is closed
under direct limits. The goal of this section is to give a different
proof of this. 

The following lemma is inspired in the proof of \cite[Thm. 4.8]{PS1}. 
\begin{lem}
\label{lem:sucF}Let $\mathcal{A}$ be an Ab5 abelian category and
$\t=(\mathcal{T},\mathcal{F})$ be a torsion pair in $\mathcal{A}$.
Consider a directed set $\Sigma$ and $F\in\Fun(\Sigma,\mathcal{A})$
such that $F(i)\in\mathcal{F}$ for all $i\in\Sigma$. Then, the following
statements hold true. 
\begin{enumerate}
\item For every short exact sequence $X\stackrel{\iota}{\hookrightarrow}\col_{\Sigma}F\stackrel{\pi}{\twoheadrightarrow}Y$
in $\mathcal{A}$ with $Y\in\mathcal{F}$, there is a short exact
sequence $F'\stackrel{}{\hookrightarrow}F\stackrel{}{\twoheadrightarrow}F''$
in $\Fun(\Sigma,\mathcal{A})$ such that: (1) $F'(i),F''(i)\in\mathcal{F}$
for all $i\in\Sigma$; (2) $X=\col_{\Sigma}F'$; and (3) $Y=\col_{\Sigma}(F'')$. 
\item In particular, there is a short exact sequence $F'\stackrel{}{\hookrightarrow}F\stackrel{}{\twoheadrightarrow}F''$
in $\Fun(\Sigma,\mathcal{A})$ such that: (1) $F'(i),F''(i)\in\mathcal{F}$
for all $i\in\Sigma$; (2) $t(\col_{\Sigma}F)=\col_{\Sigma}F'$; and
$(1:t)(\col_{\Sigma}F)=\col_{\Sigma}(F'')$.
\end{enumerate}
\end{lem}

\begin{proof}
Define $F'(i):=\Ker(\pi\circ\rho_{i}^{F})$ for every $i\in\Sigma$.
One can prove that the family $\{F'(i)\}_{i\in\Sigma}$ defines a
functor $F'\in\Fun(\Sigma,\mathcal{A})$. Moreover, we have a short
exact sequence $F'\stackrel{}{\hookrightarrow}F\stackrel{}{\twoheadrightarrow}F''$
in $\Fun(\Sigma,\mathcal{A})$. Note that, since $F'(i)\leq F(i)$,
$F''(i)=\im(\pi\circ\rho_{i}^{F})\leq Y$ and $\mathcal{F}$ is closed
under subobjects, we have that $F'(i),F''(i)\in\mathcal{F}$ for all
$i\in\Sigma$. Lastly, since $\mathcal{A}$ is Ab5, direct limits
preserve kernels. And thus, $X=\col_{\Sigma}F'$ and $Y=\col_{\Sigma}(F'')$. 
\end{proof}
In the setting of item (b) in the lemma above, note that $\col_{\Sigma}(F)\in\mathcal{F}$
if and only of $\col_{\Sigma}F'=0$. Hence, we have the following
result. 
\begin{cor}
\label{cor:1}Let $\mathcal{A}$ be an Ab5 abelian category and $\t=(\mathcal{T},\mathcal{F})$
be a torsion pair in $\mathcal{A}$. Then, the following statements
are equivalent. 
\begin{enumerate}
\item $\mathcal{F}$ is closed under direct limits. 
\item For every directed set $\Sigma$ and a functor $F\in\Fun(\Sigma,\mathcal{A})$
such that $F(i)\in\mathcal{F}$ for all $i\in\Sigma$, we have that
$\col_{\Sigma}F=0$ if $\col_{\Sigma}F\in\mathcal{T}$. 
\end{enumerate}
\end{cor}

\begin{rem}
Consider a directed set $\Sigma$ and a functor $F\in\Fun(\Sigma,\mathcal{A})$
such that $F(i)\in\mathcal{F}$ for all $i\in\Sigma$. Observe that
$F$ induces a functor $F[1]\in\Fun(\Sigma,\mathcal{H}_{\t})$, defined
as $F[1](i):=F(i)[1]$ for all $i\in\Sigma$. In what follows we will
denote the colimit of $F[1]$ as $\col_{\Sigma}^{\mathcal{H}_{\t}}F[1]$.
Similarly, a functor $T\in\Fun(\Sigma,\mathcal{A})$ with $T(i)\in\mathcal{T}$
for all $i\in\Sigma$ defines a diagram $T[0]\in\Fun(\Sigma,\mathcal{H}_{\t})$.
\end{rem}

The following lemma is based in \cite[Prop. 4.2]{PS1}. 
\begin{lem}
\label{lem:colF}Let $\mathcal{A}$ be an Ab5 and Ab3{*} abelian category
and $\t=(\mathcal{T},\mathcal{F})$ be a torsion pair in $\mathcal{A}$
such that $\mathcal{H}_{\t}$ has Hom sets. For a directed set $\Sigma$
and $T,F\in\Fun(\Sigma,\mathcal{A})$, the following statements hold
true. 
\begin{enumerate}
\item If $F(i)\in\mathcal{F}$ for all $i\in\Sigma$, then $\col_{\Sigma}^{\mathcal{H}_{\t}}(F[1])=((1:t)\col_{\Sigma}F)[1]$. 
\item If $T(i)\in\mathcal{T}$ for all $i\in\Sigma$, then $\col_{\Sigma}^{\mathcal{H}_{\t}}(T[0])=(\col_{\Sigma}T)[0]$. 
\end{enumerate}
\end{lem}

\begin{prop}
Let $\mathcal{A}$ be an Ab5 and Ab3{*} abelian category and $\t=(\mathcal{T},\mathcal{F})$
be a torsion pair in $\mathcal{A}$ such that $\mathcal{H}_{\t}$
has Hom sets. If $\mathcal{H}_{\t}$ is Ab5, then $\mathcal{F}$ is
closed under direct limits. 
\end{prop}

\begin{proof}
Consider a directed set $\Sigma$ and a functor $F\in\Fun(\Sigma,\mathcal{A})$
such that $F(i)\in\mathcal{F}$, for all $i\in\Sigma$, and $C_{F}\in\mathcal{T}$.
On the one hand, this implies for every $i\in\Sigma$, that $\rho_{i}^{F}[1]\in\Hom_{\mathcal{D}}(F(i)[1],C_{F}[1])\cong\Ext_{\mathcal{H}_{\t}}^{1}(F(i)[1],C_{F}[0]).$
Moreover, since $\rho_{i}^{F}\circ F(\lambda)=\rho_{j}^{F}$ for every
morphism $\lambda:i\to j$ in $\Sigma$, we have that $\rho_{i}^{F}[1]\circ F(\lambda)[1]=\rho_{j}^{F}[1]$.
One can use this to build a short exact sequence $\epsilon:\;\kappa^{\Sigma}(C_{F}[0])\stackrel{}{\hookrightarrow}H\stackrel{}{\twoheadrightarrow}F[1]$
in $\Fun(\Sigma,\mathcal{A})$. On the other hand, $\col_{\Sigma}^{\mathcal{H}_{\t}}(F[1])=(1:t)C_{F}[1]=0$
by Lemma \ref{lem:colF}(a). Hence, we have that 
\[
\epsilon\in\text{Ext}_{\Fun(\Sigma,\mathcal{H}_{\t})}^{1}(F[1],\kappa^{\Sigma}(C_{F}[0]))\cong\text{Ext}_{\mathcal{H}_{\t}}^{1}(\col_{\Sigma}^{\mathcal{H}_{\t}}(F[1]),C_{F}[0])=0
\]
 where the isomorphism is the one given by Theorem \ref{teo:secondmain}:
\[
\Psi_{F[1],C_{F}[0]}^{\Sigma}:\text{Ext}_{\mathcal{H}_{\t}}^{1}(\col_{\Sigma}^{\mathcal{H}_{\t}}(F[1]),C_{F}[0])\to\text{Ext}_{\Fun(\Sigma,\mathcal{H}_{\t})}^{1}(F[1],\kappa^{\Sigma}(C_{F}[0])).
\]
Hence, it follows that $\rho_{i}^{F}=0$ for all $i\in\Sigma$. And
thus, $C_{F}=0$. Therefore, $\mathcal{F}$ is closed under direct
limits by Corollary \ref{cor:1}. 
\end{proof}

\end{document}